\documentclass[12pt,a4paper]{article}
\usepackage{amsmath,amssymb,amscd,amsthm}

\newtheorem{theo}{ {\bf Theorem} }
\newtheorem{lemm}[theo]{ {\bf Lemma} }

\newtheorem{exam}[theo]{ {\bf Example} }

\newcommand{\ZZ}{{\mathbb Z}}
\newcommand{\QQ}{{\mathbb Q}}
\newcommand{\RR}{{\mathbb R}}
\newcommand{\CC}{{\mathbb C}}
\newcommand{\HH}{{\mathbb H}}

\newcommand{\ii}{{i}}

\newcommand{\im}{\mathop{\rm{Im}}\nolimits}
\newcommand{\re}{\mathop{\rm{Re}}\nolimits}
\newcommand{\etp}{{\mathbf{e}}}

\newcommand{\rmnz}{\mathop{\zeta_{\rm{R}}}\nolimits}

\newcommand{\abcd}{\begin{pmatrix}a&\!\!b\\c&\!\!d\end{pmatrix}}
\newcommand{\unmt}{\begin{pmatrix}1&\!\!0\\0&\!\!1\end{pmatrix}}

\newcommand{\mats}{\begin{pmatrix}0&\!\!-1\\1&\!\!0\end{pmatrix}}
\newcommand{\SL}{{\mathrm{SL}}}

\newcommand{\ala}{{\mathrm{A}}_2}
\newcommand{\alb}{{\mathrm{B}}_2}
\newcommand{\alg}{{\mathrm{G}}_2}

\begin{document}


\begin{center}
{\Large\bf{Modular forms from the Weierstarss functions}}
\end{center}

\bigskip

\begin{center}
 {\large
 Hiroki Aoki\footnote{
 Department of Mathematics, Faculty of Science and Technology,
 Tokyo University of Science, Noda, Chiba, 278-8510, Japan.
 aoki{\_}hiroki@ma.noda.tus.ac.jp
 }
 \quad
 and
 \quad
 Kyoji Saito\footnote{
 Institute for Physics and Mathematics of the Universe,
 the University of Tokyo, 5-1-5 Kashiwanoha, Kashiwa, 277-8583, Japan.
 kyoji.saito@ipmu.jp
 }\footnote{
 Research Institute for Mathematical Sciences, Kyoto University,
 Sakyoku Kitashirakawa, Kyoto, 606-8502, Japan.
 }\footnote{
 Laboratory of AGHA, Moscow Institute of Physics and Technology,
 9 Institutskiy per., Dolgoprudny, Moscow Region, 141700, Russian Federation.
 }
 }
\end{center}

\bigskip

\begin{abstract}
We construct holomorphic elliptic modular forms
of weight \( 2 \) and weight \( 1 \),
by special values of Weierstrass \( \wp \)-functions,
and
by differences of special values of Weierstrass
\( \zeta \)-functions,
respectively.
Also we calculated the values of these forms at
some cusps.
\end{abstract}




\section*{Introduction}

In the study of Jacobi inversion problem for the period maps
associated with primitive forms
of types \( \ala , \alb \) and \( \alg \),
the second named author has introduced a concept of
Eisenstein series of types \( \ala ,
\alb \) and \( \alg \) (cf.~\cite[{\S}8]{Sa1}).

First,
Eisenstein series of type \( \ala \) are
nothing but the classical Eisenstein series.
In this case
their weights are always equal or greater than \( 3 \).

Second,
Eisenstein series of types \( \alb\) and \( \alg \),
for the case when their weights are equal or greater than \( 3 \),
are described
by shifted classical Eisenstein series~\cite[{\S}8]{Sa1}.
Their holomorphicity at cusps and the values at cusps
can be shown and calculated similar
to the classical Eisenstein series
by helps of Riemann's zeta-function or Dirichlet's L-functions.

Third,
Eisenstein series of types \( \alb\) and \( \alg \),
for the case when their weights are equal or less than \( 2 \),
have completely different expressions.
The weight \( 2 \) Eisenstein series
of types \( \alb \) and \( \alg \) have
the expressions as special values
of Weierstrass \( \wp \)-functions.
The weight \( 1 \) Eisenstein series
of type \( \alg \) has
the expression as a difference of special values
of Weierstrass \( \zeta \)-functions.

All of these Eisenstein series
of types \( \ala \), \( \alb\) and \( \alg \),
should be elliptic modular forms,
due to the theory of period maps introduced by the second author.
Actually, in the first and second cases,
it is also clear by their expressions.
However, in the third case,
it is not so obvious.
In the present paper,
we give a short way to construct modular forms
of weight \( 2 \) and \( 1 \), including the above third case.


\section{From the Weierstrass \( \wp \)-function}

In this section we construct
elliptic modular forms of weight \( 2 \) from
the Weierstrass \( \wp \)-function.
Since the Weierstrass \( \wp \)-function
can be recognized as a meromorphic Jacobi form
of weight \( 2 \) and index \( 0 \),
its special value at \( z= s \tau +t \;
( \; (s,t) \in \QQ^2 - \ZZ^2 \;) \) has
a modularity.
Therefore, it turns out to be an elliptic modular form,
by showing the holomorphicity at each cusps.
It has been already done
by direct calculation (cf.~\cite[{\S}4.6]{DS1}),
however, here we give a proof by using Jacobi forms.

\subsection{Definition and notation}

Let \( \Omega \subset \CC \) be a \( \ZZ \)-module
generated by two \( \RR \)-linearly independent elements.
The Weierstrass \( \wp \)-function is defined by
\begin{align*}
 \wp ( \Omega ,z)
 & :=
 \frac{1}{z^2} + \sum_{\omega \in \Omega - \{ 0 \} } \left(
 \frac{1}{(z- \omega )^2} - \frac{1}{\omega^2} \right)
 \\
 & =
 \frac{1}{z^2} + \sum_{\omega \in \Omega - \{ 0 \} }
 \frac{(2\omega-z)z}{(\omega-z)^2\omega^2}
 \\
 & =
 \frac{1}{z^2} + \sum_{\omega \in \Omega - \{ 0 \} }
 \frac{\left(2-\frac{z}{\omega}\right)z}{\left(1-\frac{z}{\omega}\right)^2}
 \cdot \frac{1}{\omega^3}
 .
\end{align*}
For a while we fix \( \Omega \).
Since the sum in the third line of the above definition converges
absolutely and locally uniformly
with respect to \( z \), \( \wp ( \Omega ,z) \) is
a meromorphic function on \( z \) in \( \CC \).
The set of all poles of \( \wp ( \Omega ,z) \) is \( \Omega \) and
the order at each pole is \( 2 \).
Also, it is doubly periodic, namely,
\[
 \wp ( \Omega ,z) = \wp ( \Omega ,z+ \omega )
 \quad ( \omega \in \Omega )
 .
\]

\bigskip

In this paper we move \( \Omega \) as well as \( z \).
It is easy to see that
\[
 \wp ( \Omega ,z) = j^2 \wp (j \Omega ,jz)
 \quad ( j \in \CC - \{ 0 \} )
 .
\]
Here we put
\[
 \wp ( \tau ,z ) := \wp ( \tau \ZZ + \ZZ , z)
 .
\]
By a similar argument as above, we know
that \( \wp \) is a meromorphic function on \( \HH \times \CC \),
where we denote the complex upper half plane by
\[
 \HH := \left\{ \; \tau \in \CC \; \middle| \; \im \tau >0 \; \right\}
 .
\]
The set of all poles of \( \wp ( \tau ,z ) \) is \(
\{ \; (\tau,z) \; | \; \exists s,t \in \ZZ \text{ s.t. } z= s \tau +t \; \} \).

\subsection{Construction of elliptic modular forms}

For \( (s,t) \in \QQ^2 - \ZZ^2 \), we define
a holomorphic function on \( \HH \) by
\[
 f_{(s,t)} ( \tau ) := \wp ( \tau , s \tau +t)
 .
\]
Then the following lemma holds.
\begin{lemm}
\label{lemm:fsta}
We have
\[
 \left( f_{(s,t)} |_2 A \right) ( \tau ) :=
 (c \tau +d)^{-2} f_{(s,t)} \left( \frac{a \tau +b}{c \tau +d} \right)
 = f_{(s,t)A} ( \tau )
\]
for any \( A = \abcd \in \SL (2, \ZZ ) \).
\end{lemm}
\proof
\begin{align*}
 \left( f_{(s,t)} |_2 A \right) ( \tau )
 & :=
 (c \tau +d)^{-2} f_{(s,t)} \left( \frac{a \tau +b}{c \tau +d} \right)
 \\
 & =
 (c \tau +d)^{-2} \wp \left( \frac{a \tau +b}{c \tau +d} ,
 s \frac{a \tau +b}{c \tau +d} +t \right)
 \\
 & =
 (c \tau +d)^{-2} \wp \left( \frac{a \tau +b}{c \tau +d} \ZZ + \ZZ ,
 s \frac{a \tau +b}{c \tau +d} +t \right)
 \\
 & =
 \wp \bigl( (a \tau +b) \ZZ + (c \tau +d) \ZZ ,
 s(a \tau +b)+t(c \tau +d) \bigr)
 \\
 & =
 \wp \bigl( \tau \ZZ + \ZZ , (sa+tc) \tau + (sb+td) \bigr)
 \\
 & =
 \wp \bigl( \tau , (sa+tc) \tau + (sb+td) \bigr)
 \\
 & =
 f_{(s,t)A} ( \tau )
 .
\end{align*}
\qed

\noindent
Hence, especially, we have \( f_{(s,t)} |_2 A = f _{(s,t)} \) for
any \( A \in \Gamma_{(s,t)} \), where we denote by
\[
 \Gamma_{(s,t)} :=
 \left\{ \; A \in \SL (2, \ZZ ) \; \middle| \;
 (s,t)A-(s,t) \in \ZZ^2 \; \right\}
 .
\]
We remark that \( \Gamma_{(s,t)} \) is an elliptic modular group,
since \( \Gamma_{(s,t)} \) contains
the principal congruence subgroup of level \( L \)
\[
 \Gamma (L) :=
 \left\{ \; A \in \SL (2, \ZZ ) \; \middle| \;
 A \equiv \unmt \pmod L \; \right\}
 ,
\]
where \( L\) is the common denominator of \( s \) and \( t\).

\bigskip

In 1985, Eichler and Zagier mentioned that
the function \( \wp ( \tau ,z) \) is a meromorphic
Jacobi form of weight \( 2 \) and index \( 1 \).
Namely, in their book \cite{EZ1},
they gave the explicit formula of \( \wp ( \tau ,z) \)
as a quotient of holomorphic Jacobi forms \cite[Theorem. 3.6. (p.39)]{EZ1}:
\begin{align*}
 - \frac{3}{\pi^2} \wp ( \tau ,z)
 & =
 \frac{ \phi_{12,1} ( \tau ,z) }{ \phi_{10,1} ( \tau ,z) }
 \\
 & =
 \frac{\etp(z)+10+\etp(-z)}{\etp(z)-2+\etp(-z)} + 12
 \bigl( \etp(z)-2+\etp(-z) \bigr) \etp ( \tau ) + \cdots
 ,
\end{align*}
where \( \etp(*) := \exp \left( 2 \pi \ii * \right) \).
This series converges absolutely and locally uniformly
in \(\{ ( \tau ,z ) \in \HH \times \CC \; | \; | \im z | < \im \tau \}
\) (cf.~\cite[{\S}3.2.]{Ao1}).
Hence we have
\[
 \lim_{ \tau \to \ii \infty} f_{(s,t)} ( \tau )
 =
 \left\{
 \begin{array}{ll}
 - \dfrac{\pi^2}{3} \cdot
 \dfrac{\etp(t)+10+\etp(-t)}{\etp(t)-2+\etp(-t)} & \quad (s=0 ) \\[6mm]
 - \dfrac{\pi^2}{3} & \quad ( 0<|s|<1 )
 \end{array}
 \right.
 .
\]
Since \( f_{(s+m,t)} = f_{(s,t)} \) for any \( m \in \ZZ \),
we know that \( f_{(s,t)} \) is
holomorphic at \( \ii \infty \) for any \( (s,t) \in \QQ^2 - \ZZ^2 \),
namely, we have
\begin{equation}
\label{eq:fstc}
 \lim_{ \tau \to \ii \infty} f_{(s,t)} ( \tau )
 =
 \left\{
 \begin{array}{ll}
 - \dfrac{\pi^2}{3} \cdot
 \dfrac{\etp(t)+10+\etp(-t)}{\etp(t)-2+\etp(-t)} & \quad (s \in \ZZ ) \\[6mm]
 - \dfrac{\pi^2}{3} & \quad (s \not\in \ZZ )
 \end{array}
 \right.
 .
\end{equation}
Therefore, we have the following theorem.
\begin{theo}
\( f_{(s,t)} \) is an elliptic modular form of weight \( 2 \) with
respect to \( \Gamma_{(s,t)} \).
\end{theo}

\begin{exam}
The following two forms appear as Eisenstein series
of types \( \alb \) and \( \alg \) in Saito~\cite[{\S}8]{Sa1}, respectively.
(Eisenstein series appears in Saito~\cite{Sa1} should be distinguished
from the usual Eisenstein series appears in the theory of modular forms.)
\begin{itemize}
\item
\( f_{(0,\frac{1}{2})} \), which appears as \(
\omega_0^2 \wp \left( \frac{1}{2} \omega_0 \right) \) in Saito~\cite{Sa1},
is an elliptic modular form
of weight \( 2 \) with respect to \( \Gamma_0 (2) \). We have
\[
 \lim_{\tau \to \infty}
 f_{(0,\frac{1}{2})} ( \tau )=
 \frac{2}{3} \pi^2
 , \qquad
 \lim_{\tau \to \infty}
 \left( f_{(0,\frac{1}{2})} |_2 S \right) ( \tau )=
 - \frac{1}{3} \pi^2
 .
\]
Hence we have \( f_{(0,\frac{1}{2})} = \frac{2}{3} \pi^2 \alpha_2 \),
which is an unique modular forms of weight \( 2 \) with respect to \( \Gamma_0 (2) \) up
to constant multiplier.
\item
\( f_{(0,\frac{1}{3})} \), which appears as \(
\omega_0^2 \wp \left( \frac{1}{3} \omega_0 \right) \) in Saito~\cite{Sa1},
is an elliptic modular form
of weight \( 2 \) with respect to \( \Gamma_0 (3) \). We have
\[
 \lim_{\tau \to \infty}
 f_{(0,\frac{1}{3})} ( \tau )=
 \pi^2
 , \qquad
 \lim_{\tau \to \infty}
 \left( f_{(0,\frac{1}{3})} |_2 S \right) ( \tau )=
 - \frac{1}{3} \pi^2
 .
\]
Hence \( f_{(0,\frac{1}{3})} = \pi^2 \alpha_1^2 \),
which is an unique modular forms of weight \( 2 \) with respect to \( \Gamma_0 (3) \) up
to constant multiplier.
\end{itemize}
Here we put \( S:= \mats \) and
denote the Hecke congruence subgroup of level \( L \) by
\[
 \Gamma_0 (L) :=
 \left\{ \; A = \abcd \in \SL (2, \ZZ ) \; \middle| \;
 b \equiv 0 \pmod L
 \; \right\}
 .
\]
Notation \( \alpha_2 \) and \( \alpha_1 \) correspond
to \( \alpha \) which appears in Aoki-Ibukiyama~\cite[{\S}6]{AI1} as elliptic
modular forms of weight \( 2 \) and level \(2 \) and weight \( 1 \) and level \( 3 \),
respectively.
\end{exam}

\subsection{Values at cusps}

Although we have already known the values of \( f_{(s,t)} \) at
all cusps in previous subsection,
here we calculate them directly from the definition
without using the expression of \( \wp \) by Jacobi forms.
It is much easier than the calculation
of the Fourier expansion of \( f_{(s,t)} \),
essentially given in~\cite[{\S}4.6]{DS1}.

\bigskip

By Lemma \ref{lemm:fsta}, it is enough to calculate
\[
 \lim_{ \tau \to \ii \infty } f_{(s,t)} ( \tau )
\]
for any fixed \( 0 \leqq s<1 \) and \( 0 \leqq t<1 \).
Let \( L \) be a common denominator of \( s \) and \( t \).
Since \( f_{(s,t)} ( \tau +L ) = f_{(s,t)}( \tau ) \),
we may assume \( \re (\tau) <L \).
Also we assume \( \im ( \tau ) >L \).
Recall that
\[
 f_{(s,t)} ( \tau ) = \frac{1}{z(\tau)^2} +
 \sum_{(c,d) \in \ZZ^2 - \{(0,0)\}}
 \frac{
 \left(2-\frac{z(\tau)}{\omega(\tau)}\right)z(\tau)
 }{
 \left(1-\frac{z(\tau)}{\omega(\tau)}\right)^2
 }
 \cdot \frac{1}{\omega(\tau)^3}
 ,
\]
where \( \omega(\tau) :=c \tau +d \) and \( z(\tau):=s \tau +t \),
converges absolutely.
We decompose it as
\[
 f_{(s,t)} ( \tau ) = \frac{1}{z(\tau)^2}
 +
 \sum_{d \in \ZZ - \{ 0 \}}
 \frac{
 \left(2-\frac{z(\tau)}{d}\right)z(\tau)
 }{
 \left(1-\frac{z(\tau)}{d}\right)^2
 }
 \cdot \frac{1}{d^3}
 +
 f_{(s,t)}^* ( \tau )
 ,
\]
where we put
\[
 f_{(s,t)}^* ( \tau ) :=
 \sum_{\substack{c \in \ZZ - \{ 0 \} \\ d \in \ZZ }}
 \frac{
 \left(2-\frac{z(\tau)}{\omega(\tau)}\right)z(\tau)
 }{
 \left(1-\frac{z(\tau)}{\omega(\tau)}\right)^2
 }
 \cdot \frac{1}{\omega(\tau)^3}
 .
\]
First we show
that \( \lim_{ \tau \to \ii \infty } f_{(s,t)}^* ( \tau ) =0 \).
Since
\[
 \left| \frac{z(\tau)}{\omega(\tau)} \right|
 <
 \frac{1}{2}
 \quad \text{ for any } \tau
\]
except for finitely many \( (c,d) \),
we have
\[
 \left| f_{(s,t)}^* ( \tau ) -
 \sum_{\text{finite}}
 \frac{
 \left(2-\frac{z(\tau)}{\omega(\tau)}\right)z(\tau)
 }{
 \left(1-\frac{z(\tau)}{\omega(\tau)}\right)^2
 }
 \cdot \frac{1}{\omega(\tau)^3}
 \right| <
 10 \! \! \! \sum_{\substack{c \in \ZZ - \{ 0 \} \\ d \in \ZZ }}
 \left| \frac{z(\tau)}{\omega(\tau)^3} \right|
 .
\]
Here we use the following lemma.
\begin{lemm}
\label{lemm:eies}
Let \( k \geqq 3 \) be an integer.
Then we have
\[
 \sum_{\substack{c \in \ZZ - \{ 0 \} \\ d \in \ZZ }}
 \frac{1}{\left| \omega(\tau) \right|^k}
 <
 \frac{4}{( \im \tau )^k} \zeta_R (k)
 +
 \frac{2 \pi }{( \im \tau )^{k-1}} \rmnz (k-1)
 ,
\]
where we denote the Riemann's zeta-function by \( \zeta_R \).
\end{lemm}
\begin{proof}
\begin{align*}
 \sum_{\substack{c \in \ZZ - \{ 0 \} \\ d \in \ZZ }}
 \frac{1}{\left| \omega(\tau) \right|^k}
 &
 =
 2 \sum_{c=1}^\infty \sum_{d \in \ZZ}
 \frac{1}{\left| c \tau +d \right|^k}
 \\
 &
 <
 4 \sum_{c=1}^\infty \left(
 \frac{1}{(c \im \tau )^k} + \int_0^{\infty}
 \frac{dx}{\left( \sqrt{ (c \im \tau)^2 + x^2 } \right)^k}
 \right)
 \\
 &
 =
 4 \sum_{c=1}^\infty \left(
 \frac{1}{(c \im \tau )^k} + \int_0^{\frac{\pi}{2}}
 \frac{\left( \cos \theta \right)^{k-2} }{(c \im \tau )^{k-1}} d \theta
 \right)
 \\
 &
 <
 4 \sum_{c=1}^\infty \left(
 \frac{1}{(c \im \tau )^k} + \frac{\pi}{2}
 \frac{1}{(c \im \tau )^{k-1}}
 \right)
 \\
 &
 =
 \frac{4}{( \im \tau )^k} \zeta_R (k)
 +
 \frac{2 \pi }{( \im \tau )^{k-1}} \rmnz (k-1)
\end{align*}
\end{proof}
By using this lemma, we have
\[
 \lim_{ \tau \to \ii \infty }
 f_{(s,t)}^* ( \tau ) =
 \sum_{\text{finite}}
 \lim_{ \tau \to \ii \infty }
 \left( \frac{
 \left(2-\frac{z(\tau)}{\omega(\tau)}\right)z(\tau)
 }{
 \left(1-\frac{z(\tau)}{\omega(\tau)}\right)^2
 }
 \cdot \frac{1}{\omega(\tau)^3}
 \right)
 =0
 .
\]
Therefore, we have
\begin{equation}
 \label{eq:fsts}
 \lim_{ \tau \to \ii \infty }
 f_{(s,t)} ( \tau )
 =
 \lim_{ \tau \to \ii \infty }
 \frac{1}{z(\tau)^2}
 +
 \lim_{ \tau \to \ii \infty }
 \sum_{d \in \ZZ - \{ 0 \}}
 \frac{
 \left(2-\frac{z(\tau)}{d}\right)z(\tau)
 }{
 \left(1-\frac{z(\tau)}{d}\right)^2
 }
 \cdot \frac{1}{d^3}
 .
\end{equation}

\bigskip

\noindent
{\bf{(Case: \( s \neq 0 \))}} \\
From (\ref{eq:fsts}), we have
\[
 \lim_{ \tau \to \ii \infty } f_{(s,t)} ( \tau )
 =
 \sum_{d \in \ZZ - \{ 0 \} } \frac{-1}{ d^2 }
 = -2 \rmnz (2)
 .
\]
Comparing with (\ref{eq:fstc}),
we have the famous formula
\begin{equation}
 \label{eq:rmnt}
 \rmnz (2) = \frac{\pi^2}{6}
 .
\end{equation}
Since Jacobi forms \( \phi_{12,1} \) and \( \phi_{10,1} \) are
constructed from (Jacobi or lattice) theta-functions,
without using the Riemann's
zeta-function, this is a new proof of (\ref{eq:rmnt}).

\bigskip

\noindent
{\bf{(Case: \( s=0 \))}} \\
From (\ref{eq:fsts}), we have
\begin{align*}
 \lim_{ \tau \to \ii \infty } f_{(0,t)} ( \tau )
 & =
 \frac{1}{t^2} +
 \sum_{d \in \ZZ - \{ 0 \} }
 \frac{
 \left(2-\frac{t}{d}\right)t
 }{
 \left(1-\frac{t}{d}\right)^2
 }
 \cdot \frac{1}{d^3}
 \\
 & =
 \frac{1}{t^2} +
 \sum_{d \in \ZZ - \{ 0 \} }
 \frac{t}{d^3}
 \sum_{n=0}^{\infty}
 (n+2) \left( \frac{t}{d} \right)^n
 \\
 & =
 \frac{1}{t^2} + 2
 \sum_{n=1}^{\infty}
 \sum_{d=1}^{\infty}
 (2n+1) \frac{t^{2n}}{d^{2n+2}}
 \\
 & =
 \frac{1}{t^2} + 2 \sum_{n=1}^{\infty} (2n+1) \rmnz (2n+2) t^{2n}
 .
\end{align*}
Hence we have
\begin{equation}
 \label{eq:rmnr}
 \frac{1}{t^2} + 2 \sum_{n=1}^{\infty} (2n+1) \rmnz (2n+2) t^{2n}
 =
 - \dfrac{\pi^2}{3} \cdot
 \dfrac{\etp(t)+10+\etp(-t)}{\etp(t)-2+\etp(-t)}
 .
\end{equation}
This is a relation between the special values of
the Riemann's zeta-function.
We remark that the equation (\ref{eq:rmnr}) can be shown
by using the Bernoulli numbers \( B_{2n} \):
\[
 \frac{1}{2}z + \frac{z}{e^z -1}= \sum_{n=0}^{\infty} \frac{B_{2n}}{(2n)!} z^{2n}
 ,
 \qquad
 \rmnz(2n+2) = (-1)^n \frac{B_{2n+2} (2 \pi )^{2n+2} }{2 (2n+2)!}
 .
\]


\section{From the Weierstrass \( \zeta \)-function}

In this section we construct
elliptic modular forms of weight \( 1 \) from
the Weierstrass \( \zeta \)-function.
Since the Weierstrass \( \zeta \)-function
is quasi-periodic and is not doubly periodic,
its special value itself does not have
a modularity.
Classically, to gain a modularity,
we modify it to the Hecke form
by decreasing
the Weierstrass \( \eta \)-function (cf.~\cite[Chapter 15]{La1}
or~\cite[{\S}4.8]{DS1}).
However, here we can construct an elliptic modular form
taking the difference of two special values properly.
Although this is a corollary of the classical argument,
here we calculate the values at some cusps
by much easier calculation.

\subsection{Definition and notation}

Again let \( \Omega \subset \CC \) be a \( \ZZ \)-module
generated by two \( \RR \)-linearly independent elements.
The Weierstrass \( \zeta \)-function is defined by
\begin{align*}
 \zeta ( \Omega ,z)
 & :=
 \frac{1}{z} + \sum_{\omega \in \Omega - \{ 0 \} } \left(
 \frac{1}{z- \omega} + \frac{1}{\omega} + \frac{z}{\omega^2} \right)
 \\
 & :=
 \frac{1}{z} - \sum_{\omega \in \Omega - \{ 0 \} }
 \frac{z^2}{(\omega-z)\omega^2}
 \\
 & :=
 \frac{1}{z} - \sum_{\omega \in \Omega - \{ 0 \} }
 \frac{z^2}{\left(1-\frac{z}{\omega}\right)}
 \cdot \frac{1}{\omega^3}
 .
\end{align*}
For a while we fix \( \Omega \).
Since the sum in the third line of the above definition converges
absolutely and locally uniformly
with respect to \( z \), \( \zeta ( \Omega ,z) \) is
a meromorphic function on \( z \) in \( \CC \).
The set of all poles of \( \zeta ( \Omega ,z) \) is \( \Omega \) and
the order at each pole is \( 1 \).
It is not doubly periodic,
however, \( \zeta ( \Omega ,z+ \omega ) - \zeta ( \Omega ,z) \) does
not depend on \( z \), but only on \( \omega \).

\bigskip

In this paper we move \( \Omega \) as well as \( z \).
It is easy to see that
\[
 \zeta ( \Omega ,z) = j \zeta (j \Omega ,jz)
 \quad ( j \in \CC - \{ 0 \} )
 .
\]
Here we put
\[
 \zeta ( \tau ,z ) := \zeta ( \tau \ZZ + \ZZ , z)
 .
\]
By a similar argument as above, we know
that \( \zeta \) is a meromorphic function on \( \HH \times \CC \).
The set of all poles of \( \zeta ( \tau ,z ) \) is \(
\{ \; (\tau,z) \; | \; \exists s,t \in \ZZ \text{ s.t. } z= s \tau +t \; \} \).
We define \( \eta_1 ( \tau ) \) and \( \eta_2 ( \tau) \) by
\[
 \zeta ( \tau , z+ \tau ) - \zeta ( \tau , z) = \eta_1 ( \tau )
 \quad \text{ and } \quad
 \zeta ( \tau , z+1 ) - \zeta ( \tau , z) = \eta_2 ( \tau )
 .
\]

\subsection{Construction of an elliptic modular form}

For \( (s,t) \in \QQ^2 - \ZZ^2 \), we define
a holomorphic function on \( \HH \) by
\[
 g_{(s,t)} ( \tau ) := \zeta ( \tau , s \tau +t)
 .
\]
Then the following lemma holds.
\begin{lemm}
\label{lemm:gsta}
We have
\[
 \left( g_{(s,t)} |_1 A \right) ( \tau ) :=
 (c \tau +d)^{-1} g_{(s,t)} \left( \frac{a \tau +b}{c \tau +d} \right)
 = g_{(s,t)A} ( \tau )
\]
for any \( A = \abcd \in \SL (2, \ZZ ) \).
\end{lemm}
\proof
\begin{align*}
 \left( g_{(s,t)} |_1 A \right) ( \tau )
 & :=
 (c \tau +d)^{-1} g_{(s,t)} \left( \frac{a \tau +b}{c \tau +d} \right)
 \\
 & =
 (c \tau +d)^{-1} \zeta \left( \frac{a \tau +b}{c \tau +d} ,
 s \frac{a \tau +b}{c \tau +d} +t \right)
 \\
 & =
 (c \tau +d)^{-1} \zeta \left( \frac{a \tau +b}{c \tau +d} \ZZ + \ZZ ,
 s \frac{a \tau +b}{c \tau +d} +t \right)
 \\
 & =
 \zeta \bigl( (a \tau +b) \ZZ + (c \tau +d) \ZZ ,
 s(a \tau +b)+t(c \tau +d) \bigr)
 \\
 & =
 \zeta \bigl( \tau \ZZ + \ZZ , (sa+tc) \tau + (sb+td) \bigr)
 \\
 & =
 \zeta \bigl( \tau , (sa+tc) \tau + (sb+td) \bigr)
 \\
 & =
 g_{(s,t)A} ( \tau )
 .
\end{align*}
\qed

\noindent
However, \( g_{(s,t)} \) itself is not modular, since we have
\begin{equation}
\label{eq:gstt}
 \left( g_{(s,t)} |_1 A \right) ( \tau ) -g_{(s,t)} ( \tau )
 =
 g_{(s,t)A} ( \tau ) - g_{(s,t)} ( \tau )
 =
 u \eta_1 ( \tau ) + v \eta_2 ( \tau )
\end{equation}
for any \( A \in \Gamma_{(s,t)} \),
where \( u:= s(a-1)+tc \in \ZZ \) and \( v:= sb+t(d-1) \in \ZZ \).

\bigskip

Now we take \( r \in \ZZ - \{ 0 \} \) such that \( (rs,rt) \not\in \ZZ^2 \).
We remark that \( \Gamma_{(s,t)} \subset \Gamma_{(rs,rt)} \).
Let
\[
 h_{r,(s,t) } ( \tau ) := r g_{(s,t)} ( \tau ) - g_{(rs,rt)} ( \tau )
 .
\]
Then, from (\ref{eq:gstt}), \( h_{r,(s,t)} \) has
the automorphic property of weight \( 1 \) with
respect to \( \Gamma_{(s,t)} \), namely,
we have \( h_{r,(s,t)} |_1 A = h_{r,(s,t)} \) for
any \( A \in \Gamma_{(s,t)} \).
The following lemma holds.
\begin{lemm}
\label{lemm:hrst}
\( h_{r,(s,t)} \) is bounded at each cusp.
\end{lemm}
We give a proof of this lemma in the next subsection.
Consequently, we have the following theorem.
\begin{theo}
\label{theo:hrst}
\( h_{r,(s,t)} \) is an elliptic modular form of weight \( 1 \) with
respect to \( \Gamma_{(s,t)} \).
\end{theo}
More generally, the following theorem holds.
\begin{theo}
Let \( U:= \left\{ u_1 , u_2 , \dots , u_m \right\} \) be
a set of \( m \) tuples, where \( u_j := \left( s_j , t_j \right)
\in \QQ^2 - \ZZ^2 \).
We put
\[
 h_U ( \tau ) := \sum_{j=1}^m g_{u_j} ( \tau )
\]
and
\[
 \Gamma_U := \bigcap_{j=1}^m \Gamma_{u_j}
 .
\]
If \( \sum_{j=1}^m u_j =(0,0) \),
then \( h_U \) is an elliptic modular form of weight \( 1 \) with
respect to \( \Gamma_U \)
\end{theo}
Here we remark that \( g_{(-s,-t)} ( \tau ) = - g_{(s,t)} ( \tau ) \).
Therefore, this theorem contains Theorem \ref{theo:hrst}.
This theorem can be shown by the same way as Theorem \ref{theo:hrst}.

\subsection{Values at cusps}

In this subsection, we give a proof of Lemma \ref{lemm:hrst},
in a similar manner as subsection 1.3.
Here we fix \( s \) and \( t \).
By the argument in the previous section,
we may assume that \( 0 \leqq s<1 \) and \( 0 \leqq t<1 \).
By Lemma \ref{lemm:gsta},
it is enough to show
that \( h_{r,(s,t)} ( \tau ) \) is bounded for
sufficiently large \( \im \tau \).
Let \( L \) be a common denominator of \( s \) and \( t \).
Since \( h_{r,(s,t)} ( \tau +L ) = h_{r,(s,t)}( \tau ) \),
we may assume \( \re (\tau) <L \).
Also we assume \( \im ( \tau ) >L \).
From the definition, we have
\begin{equation}
\label{eq:hsts}
 h_{r,(s,t)} ( \tau )
 =
 \frac{r^2-1}{rz(\tau)}
 + \! \! \! \!
 \sum_{(c,d) \in \ZZ^2 - \{(0,0)\}}
 \frac{r(r-1)z(\tau)^2}{
 \left(1- \frac{z(\tau)}{\omega(\tau)} \right)
 \left(1- \frac{rz(\tau)}{\omega(\tau)} \right) }
 \cdot \frac{1}{\omega(\tau)^3}
 ,
\end{equation}
where \( \omega(\tau) :=c \tau +d \) and \( z(\tau):=s \tau +t \).
The sum in (\ref{eq:hsts}) converges absolutely.
We decompose it as
\[
 h_{r,(s,t)} ( \tau ) = \frac{r^2-1}{rz(\tau)}
 +
 \sum_{d \in \ZZ - \{ 0 \}}
\frac{r(r-1)z(\tau)^2}{
 \left(1- \frac{z(\tau)}{d} \right)
 \left(1- \frac{rz(\tau)}{d} \right) }
 \cdot \frac{1}{d^3}
 +
 h_{r,(s,t)}^* ( \tau )
 ,
\]
where we put
\[
 h_{r,(s,t)}^* ( \tau ) :=
 \sum_{\substack{c \in \ZZ - \{ 0 \} \\ d \in \ZZ }}
 \frac{r(r-1)z(\tau)^2}{
 \left(1- \frac{z(\tau)}{\omega(\tau)} \right)
 \left(1- \frac{rz(\tau)}{\omega(\tau)} \right) }
 \cdot \frac{1}{\omega(\tau)^3}
 .
\]
Here we show
that \( h_{r,(s,t)}^* ( \tau ) \) is bounded for
large \( \im \tau \).
Since
\[
 \left| \frac{z(\tau)}{\omega(\tau)} \right|
 <
 \frac{1}{2}
 \quad \text{ for any } \tau
\]
except for finitely many \( (c,d) \),
we have
\[
 \left| h_{r,(s,t)}^* ( \tau ) -
 \sum_{\text{finite}}
 \frac{r(r-1)z(\tau)^2}{
 \left(1- \frac{z(\tau)}{\omega(\tau)} \right)
 \left(1- \frac{rz(\tau)}{\omega(\tau)} \right) }
 \cdot \frac{1}{\omega(\tau)^3}
 \right| <
 4r(r-1) \! \! \! \sum_{\substack{c \in \ZZ - \{ 0 \} \\ d \in \ZZ }}
 \left| \frac{z(\tau)^2}{\omega(\tau)^3} \right|
 .
\]
Hence, by using this Lemma \ref{lemm:eies}, we know
that \( h_{r,(s,t)}^* ( \tau ) \) is bounded for
large \( \im \tau \).
Therefore, \( h_{r,(s,t)} ( \tau ) \) is bounded for
large \( \im \tau \) also.

\bigskip

Although we know that \( h_{r,(s,t)} ( \tau ) \) is bounded for
large \( \im \tau \), to calculate the value at \( \ii \infty \) is
not so easy, because \( h_{r,(s,t)} \) does not vanish
at \( \ii \infty \) except when \( s=0 \).
Hereafter, we calculate the value at \( \ii \infty \) only
in the case of \( s=0 \).

\bigskip

\noindent
{\bf{(Case: \( s=0 \))}} \\
When \( s=0 \), we have
\[
 \lim_{ \tau \to \ii \infty } h_{r,(0,t)}^* ( \tau )
 =0
 .
\]
Therefore we have
\begin{align*}
 \lim_{ \tau \to \ii \infty }
 h_{r,(0,t)} ( \tau )
 &
 =
 \frac{r^2-1}{rt} +
 \sum_{d \in \ZZ - \{ 0 \}}
 \frac{r(r-1)t^2}{
 \left(1- \frac{t}{d} \right)
 \left(1- \frac{rt}{d} \right) }
 \cdot \frac{1}{d^3}
 \\
 & =
 \frac{r^2-1}{rt} + r
 \sum_{d \in \ZZ - \{ 0 \}}
 \frac{t^2}{d^3}
 \sum_{n=0}^{\infty}
 \left( r^{n+1} -1 \right)
 \left( \frac{t}{d} \right)^n
 \\
 & \hspace*{-16pt} =
 \frac{r^2-1}{rt} + 2r
 \sum_{n=1}^{\infty}
 \sum_{d=1}^{\infty}
 \left( r^{2n} -1 \right)
 \frac{t^{2n+1}}{d^{2n+2}}
 \\
 & \hspace*{-32pt} =
 \frac{r^2-1}{rt} +2r \sum_{n=1}^{\infty}
 \left( r^{2n} -1 \right) \rmnz (2n+2) t^{2n+1}
 \\
 & \hspace*{-48pt} =
 \frac{r^2-1}{rt} +r \sum_{n=1}^{\infty}
 \left( r^{2n} -1 \right)
 \frac{B_{2n+2}}{(2n+2)!} (-1)^n (2 \pi )^{2n+2} t^{2n+1}
 \\
 & \hspace*{-64pt} =
 \frac{r^2-1}{rt} +
 \left(
 - \frac{1}{rt} \sum_{n=1}^{\infty}
 \frac{B_{2n+2}}{(2n+2)!} \left( 2 \pi \ii rt \right)^{2n+2}
 + \frac{r}{t} \sum_{n=1}^{\infty}
 \frac{B_{2n+2}}{(2n+2)!} \left( 2 \pi \ii t \right)^{2n+2}
 \right)
 \\
 & \hspace*{-64pt} =
 \frac{r^2-1}{rt}
 - \frac{1}{rt}
 \left( -1+ ( \pi \ii rt ) - \frac{1}{12} (2 \pi \ii rt)^2
 + \frac{2 \pi \ii rt}{\etp(rt)-1} \right)
 \\
 & \hspace*{-32pt}
 + \frac{r}{t}
 \left( -1+ ( \pi \ii t ) - \frac{1}{12} (2 \pi \ii t)^2
 + \frac{2 \pi \ii t}{\etp(t)-1} \right)
 \\
 & \hspace*{-64pt} =
 2\pi \ii
 \left(
 \frac{r-1}{2}
 + \frac{r}{\etp(t)-1}
 - \frac{1}{\etp(rt)-1}
 \right)
 .
\end{align*}

\begin{exam}
\( h_{2,(0,\frac{1}{3})} \) appears as Eisenstein series
of type \( \alg \) in Saito~\cite[{\S}8]{Sa1}, that is, \( \omega_0
\left( 2 \zeta \left( \frac{1}{3} \omega_0 \right) -
\zeta \left( \frac{2}{3} \omega_0 \right) \right) \).
This \( h_{2,(0,\frac{1}{3})} \) is an elliptic modular form
of weight \( 1 \) with respect to \( \Gamma_1 (3) \). We have
\[
 \lim_{\tau \to \infty}
 h_{2,(0,\frac{1}{3})} ( \tau )=
 - \sqrt{3} \pi \ii
 .
\]
Hence \( h_{2,(0,\frac{1}{3})} = - \sqrt{3} \pi \ii \alpha_1 \),
which is an unique modular forms of weight \( 1 \) with
respect to \( \Gamma_1 (3) \) up to constant multiplier.
Here we put
\[
 \Gamma_1 (L) :=
 \left\{ \; A = \abcd \in \SL (2, \ZZ ) \; \middle| \;
 \begin{array}{c}
 a \equiv d \equiv 1 \pmod L , \\ b \equiv 0 \pmod L
 \end{array}
 \; \right\}
 .
\]
\end{exam}


\section*{Acknowledgements}

This work was partially supported by JSPS KAKENHI Grant Number
JP19K03429 (the first named author)
and
JP18H01116 (the second named author).


\bigskip


\bigskip

\end{document}